\documentclass[11pt,reqno,a4paper]{article}
\usepackage[english]{babel}
\usepackage{amsmath}
\usepackage{amsthm}
\usepackage{amssymb}
\usepackage{enumerate}
\usepackage{xspace}
\usepackage{euscript}
\usepackage{graphicx}
\usepackage{graphics}
\usepackage{amscd}
\usepackage{epsfig}
\usepackage{tabularx}
\usepackage[usenames]{color}

\usepackage{geometry}
\usepackage[latin1]{inputenc}
\usepackage{amsmath}
\usepackage{amsfonts}
\usepackage{amssymb}
\usepackage{graphicx}
\usepackage{epstopdf}

\newtheorem{lem}{Lemma}[section]
\newtheorem{defn}{Definition}[section]
\newtheorem{thm}[lem]{Theorem}
\newtheorem{claim}{Claim}[section]

\newtheorem{rem}[lem]{Remark}

\numberwithin{equation}{section}

\newcommand{\tx}{\tilde{x}}
\newcommand{\ty}{\tilde{y}}

\newcommand{\iu}{{i\mkern1mu}} 

\newcommand{\floor}[1]{\lfloor #1 \rfloor}

\newcommand{\LpCr}[1]{|| #1 ||_{L^p(C_r)}}

\title{\textbf{Singular solutions for complex second order elliptic equations and their application to time-harmonic diffuse optical tomography.}}
\author{\textsc{Jason Curran\thanks{Department of Mathematics and Statistics, University of Limerick, Ireland, Jason.Curran$@$ul.ie},\quad Romina
Gaburro}\thanks{Department of Mathematics and Statistics, Health Research Institute (HRI),
University of Limerick, Ireland,
Romina.Gaburro$@$ul.ie},\\ \textsc{Clifford J. Nolan\thanks{Department of Mathematics and Statistics,
Health Research Institute (HRI), University of Limerick, Ireland, Clifford.Nolan$@$ul.ie}}}

\date{}

\begin{document}
\maketitle

 \begin{abstract}
We construct singular solutions of a complex elliptic equation of second order, having an isolated singularity of any order. In particular, we extend results obtained for the real partial differential equation in divergence form by Alessandrini in 1990. Our solutions can be applied to the determination of the optical properties of an anisotropic medium in time-harmonic Diffuse Optical Tomography (DOT). 
\end{abstract}



\section{Introduction}\label{section introduction}
\setcounter{equation}{0}

In this paper we construct singular solutions to the complex elliptic equation
\begin{equation}\label{Sigma,q equation}
Lu = -\mbox{div}\left(K\nabla u\right)+qu=0,
\end{equation} 
with an isolated singularity of an arbitrary high order. The study of singular solutions of elliptic equations having isolated singularities goes back to the $1950$'s and $60$'s with the works of John \cite{J}, Bers \cite{B}, Gilbarg and Serrin \cite{G-Se} and Marcus \cite{Mar}. In 1990 Alessandrini \cite{A1} constructed singular solutions with an isolated singularity of an arbitrary high order for a real equation of type \eqref{Sigma,q equation}, with $q=0$ and $K\in W^{1,p}(\Omega)$ in a ball $B\subset\mathbb{R}^n$, $p>n$. Such solutions were applied to the so-called Calder\`on's inverse conductivity problem for the stable determination of the conductivity $K$ from the so-called Dirichlet-to-Neumann (D-N) boundary map $\Lambda_{K}$ in some domain $\Omega\subset\mathbb{R}^n$, with $n\geq 2$ (see the review papers \cite{Bo}, \cite{U} on this inverse problem). We also recall the paper of Isakov \cite{I}, who employed singular solutions to determine (the discontinuities of) the conductivity $K$ in  \eqref{Sigma,q equation} from $\Lambda_{K}$. Here we extend the construction of singular solutions in \cite{A1} to the complex equation  \eqref{Sigma,q equation}. A partial extension to the complex case, where $K$ was assumed to be a multiple of the identity matrix $I$ near the boundary $\partial\Omega$, was obtained in \cite{JRCE}. In the current work we completely remove this assumption for $K$ near the boundary.

The application we have in mind is the time-harmonic Diffuse Optical Tomography (DOT) in the presence of anisotropy up to the surface of the medium $\Omega$ under consideration. DOT is a minimally invasive imaging modality with significant potentials in medical applications. An object of unknown internal optical properties is illuminated with a near-infrared light  through its boundary and the scattered and transmitted light is measured ideally everywhere on the boundary. The optical properties of the material (the \textit{absorption} and \textit{scattering coefficients}) are then estimated from this boundary information (\cite{Ar}, \cite{Ha}, \cite{HS}, \cite{H}), typically encoded in a boundary map, e.g. the D-N map (see definition \ref{DN map}). In \cite{JRCE}, stability estimates of the derivatives (of any order) of the absorption coefficient at the boundary were established in terms of the D-N map corresponding to \eqref{Sigma,q equation}, under the assumption that the medium $\Omega$ under inspection was isotropic ($K=aI$, with $a$ scalar function) near its boundary. This limitation in the anisotropic structure of $K$ was needed for the construction of singular solutions to the complex \eqref{Sigma,q equation} with isolated singularities $y_0$, such that $K(y_0)=a(y_0)I$. Here we remove this assumption and extend the construction of singular solutions to the complex \eqref{Sigma,q equation} with no restriction on $K$ near its singularity $y_0$. This allows for the extension of the stability of the derivatives of the absorption coefficient in \cite{JRCE} to the full anisotropic DOT case in the low frequencies regime, which suits the DOT experiment we have in mind.

The paper is organised as follows. In section \ref{sec:Formulation}, we rigorously formulate the problem, state and prove our main result (Theorem \ref{theor singular sol}) on the construction of singular solutions to \eqref{Sigma,q equation} with complex coefficients. A crucial step for the machinery of the proof of Theorem \ref{theor singular sol} is Claim \ref{claim 1}. Section \ref{sec:DOTApplication} is devoted to the application of Theorem \ref{theor singular sol} to the DOT problem. We show that by feeding the singular solutions constructed in section \ref{sec:Formulation} into the so-called Alessandrini Identity \eqref{Alessandrini identity II}, H\"older type stability estimates of the derivatives (of any order) of the absorption coefficient of the medium in terms of the D-N map (see definition \ref{DN map}) can be established at the boundary, therefore extending results in \cite{JRCE} to the full anisotropic DOT problem.



\section{Formulation of the problem and main result}\label{sec:Formulation}

For $n\geq 3$, a point $x\in \mathbb{R}^n$ will be denoted by $x=(x',x_n)$, where $x'\in\mathbb{R}^{n-1}$ and $x_n\in\mathbb{R}$.
Moreover, given a point $x\in \mathbb{R}^n$, we will denote with $B_r(x), B_r'(x')$ the open balls in
$\mathbb{R}^{n},\mathbb{R}^{n-1}$, centred at $x$ and $x'$ respectively with radius $r$ and by $Q_r(x)$ the cylinder $B_r'(x')\times(x_n-r,x_n+r).$ We denote $B_r=B_r(0)$, $B'_r=B'_r(0)$ and $Q_r=Q_r(0)$. 

We consider the operator $L$ introduced in \eqref{Sigma,q equation} on $B_R(y)$, $y\in\mathbb{R}^n$. $K$ and $q$ are a complex-symmetric, matrix-valued and complex-valued function respectively. Setting $K = K_R + i K_I$, we assume that $K_RK_I = K_IK_R$, as this fits the application in section \ref{sec:DOTApplication}, and that there are positive constants $\lambda, \Lambda, E$ and $p > n$ such that 
\begin{align}
&|K(x)|\leq \lambda,\quad  |q(x)| < \Lambda,  \qquad\mbox{for any}\quad x \in B_R(y), \label{BoundednessAssumption} \\
&\lambda^{-1} |\xi|^2 \leq  K_R(x) \xi \cdot \xi , \qquad \mbox{ for a.e } x \in B_R(y), \quad\mbox{ for all } \xi \in \mathbb{R}^n,  \label{EllipticityAssumption} \\ 
&\lambda^{-1}  |\xi|^2 \leq  K_I(x) \xi \cdot \xi \text{ or }  K_I(x) \xi \cdot \xi \leq -\lambda^{-1} |\xi|^2,  \mbox{ for a.e } x \in B_R(y), \mbox{ for all } \xi \in \mathbb{R}^n, \label{EllipticityAssumption2}\\
&||K^{ij}||_{W^{1,\:p}(B_R(y))} \leq E,  \qquad\mbox{ for } i,j = 1, \ldots, n.\label{HolderAssumption}
\end{align}
Setting $q=q_R+iq_I$, \eqref{Sigma,q equation} is equivalent to the system for $u=(u^1, u^2)$
\begin{equation}\label{system compact}
-\mbox{div}(\kappa\nabla u) + qu = 0,\qquad\textnormal{in}\quad B_R(y),
\end{equation}
where $\kappa = \left(\begin{array}{rr} K_R & - K_I \\ K_I & K_R\end{array}\right)$ and $q = \left(\begin{array}{rr} q_R & - q_I \\ q_I & q_R\end{array}\right)$. Since $\kappa\xi\cdot \xi = K_R\xi_1\cdot \xi_1 + K_R\xi_2\cdot \xi_2$, for any $\xi = (\xi_1,\xi_2)\in\mathbb{R}^{2n}$, then \eqref{BoundednessAssumption}, \eqref{EllipticityAssumption} imply that system \eqref{system compact} satisfies the strong ellipticity condition 
\begin{equation}\label{strong ellipt}
\lambda^{-1}|\xi|^2\leq \kappa \xi \cdot \xi \leq \lambda |\xi|^2, \qquad\text{for\:a.e}\: x\in\Omega,\quad\text{for\:all}\quad \xi\in\mathbb{R}^{2n}.
\end{equation} 
Below we introduce some notation adopted through this paper:
 \begin{enumerate}
\item $C_{y,r} = \{x \in\mathbb{R}^n \; \vert \; r < |x-y| < 2r\}$, for $y\in\mathbb{R}^n$ and $r>0$.
\item $K_{(n)}^{-1}(y)$, $K_{nn}^{-1}(y)$ denote the last row, last entry in the last row of matrix  $K^{-1}(y)$, respectively, for $y\in\mathbb{R}^n$.
\item $\tilde{z}_{x,y} := \frac{K_{(n)}^{-1}(y)(x - y)}{\left(K_{nn}^{-1}(y)\right)^{1/2}\left( K^{-1}(y)(x-y)\cdot (x-y)\right)^{1/2}}\in\mathbb{C}$, for $x,y\in\mathbb{R}^n$.
\item For $v,w\in\mathbb{C}^n$, with $v=(v_1,\dots , v_n)$, $w=(w_1,\dots , w_n)$, we understand that $v\cdot w = \sum_{i=1}^{n} v_i w_i$.
\item $C$ will always denote a positive constant, which may vary throughout the paper.  
\item For $z\in\mathbb{C}$, we denote by  $\Re{z}$ and $\Im{z}$ the real and imaginary part of $z$, respectively.
\end{enumerate}
\begin{rem}
\eqref{EllipticityAssumption} implies that that for any $x,y\in\mathbb{R}^n$, with $x\neq y$, $\Re\left\{\left(K^{-1}(y)(x-y)\cdot (x-y)\right)\right\}>0$ and that without loss of generality we can assume that $\Re (K_{nn}^{-1}(y))>0$.
\end{rem}
Next, we state and prove our main result.

\begin{thm}\textnormal{(Singular solutions)}.\label{theor singular sol}
Given $L$ on $B_R(y_0)$ as in \eqref{Sigma,q equation} and satisfying \eqref{BoundednessAssumption} - \eqref{HolderAssumption}, for any $m=0,1,2,\dots$, there exists $u\in{W}_{loc}^{2,\:p}(B_{R}(y_0)\setminus\{y_0\})$ such that $Lu=0$ in ${B_R(y_0)\setminus\{y_0\}}$, with
\begin{equation}\label{sing solution G}
u(x)  = \left(K^{-1}(y_0)(x-y_0)\cdot (x-y_0)\right)^{\frac{2-n-m}{2}} m! \left[\left(K^{-1}(y_0)\right)_{nn}\right]^{\frac{m}{2}} C_m^{\frac{n-2}{2}}(\tilde{z}_{x,y_0}) + w(x), 
\end{equation}
where $C_m^{\frac{n-2}{2}}:\mathbb{C}\rightarrow\mathbb{C}$ is the Gegenbauer polynomial of degree $m$ and order $\frac{n-2}{2}$. Moreover, $w$ satisfies
\begin{eqnarray}
& &\vert{\:w}(x)\vert+\vert{\:x-y_0}\:\vert\:\vert{D}w(x)\vert\leq{C}\:\vert{\:x-y_0}\:\vert^
{\:2-n-m+\alpha},\quad{in}\quad{B}_{R}(y_0)\setminus\{y_0\}, \label{w stima lipschitz} \\
& &||D^{2}w||_{L^p(C_{y_0,r})} \leq{C}\:r^{\frac{n}{p}-n-m+\alpha},\quad\mbox{for}\:every\quad{r}, \:0<r<R/2. \label{w stima int}
\end{eqnarray}
Here $\alpha$ is such that $0<\alpha<\beta$, and $C$ depends only on $\alpha,\:n,\:p,\:R$, $\lambda$, $\Lambda$, and $E$. 
\end{thm}

\begin{proof}[Proof of Theorem \ref{theor singular sol}] Setting $u_m := \frac{\partial ^{m} u_0}{\partial y_n ^m} \Big\vert_{y=y_0}$, $m=0,1,2,\dots$, where $u_0(x):=\Gamma_{K(y_0)}(x-y_0)$\\$=\left(K^{-1}(y_0) (x-y_0)\cdot (x-y_0)\right)^{\frac{2-n}{2}}$ is the fundamental solution, with a pole at $y_0$, of
\begin{equation}\label{L0}
L_{0}:=\mbox{div}\left(K(y_0)\nabla\cdot\right),
\end{equation}
it was shown in \cite{JRCE}, via an inductive argument, that
\begin{equation}\label{um1}
u_m(x) = \left(K^{-1}(y_0)(x-y_0)\cdot (x-y_0)\right)^{\frac{2-n-m}{2}}m! \left[K_{nn}^{-1}(y_0)\right]^{\frac{m}{2}} C_{m}^{\frac{n-2}{2}}(\tilde{z}_{x,y_0}).
\end{equation}
Note that $L_{0} u_m = 0$, in $\mathbb{R}^n\setminus\{y_0\}$. To find $w$ solution to 
\begin{equation}\label{solution Lq}
Lw = - Lu_m, \qquad\textnormal{in}\quad B_R(y_0)\setminus\{y_0\},
\end{equation}
satisfying \eqref{w stima lipschitz}, \eqref{w stima int}, we estimate the $L^p$-norm of the right hand side of \eqref{solution Lq} as
\begin{equation}\label{Lp est 1}
\LpCr{L u_m}  \leq C r^{\frac{n}{p} -(n+m-\beta)},\qquad \beta= 1 - n/p,
\end{equation} 
(see \cite{JRCE} for details), where $C$ depends $n$, $p$, $R$, $\lambda$, $\Lambda$, and $E$.\\

We start by showing that there exists $w_0\in W^{2,p}_{loc}(B_R(y_0)\setminus\{y_0\})$ such that 
\begin{equation}\label{w0}
L_{0}w_0= -Lu_m\quad\text{and}\quad |w_0(x)|\leq C |x-y_0|^{2-(n+m-\beta)}, \quad\text{for\:any}\:x\in B_R(y_0)\setminus\{y_0\}.
\end{equation}
To prove \eqref{w0}, we need the following result, which is crucial in the machinery of our complex singular solutions' construction, as it extends results in \cite[Lemma 2.3]{A1}, \cite[Lemma 3.6]{JRCE} on the Laplace operator, to the case of a constant complex operator $L_0$.
\begin{claim}\label{claim 1}
If $f\in L^{p}_{loc}(B_R \setminus\{0\})$ is a complex-valued function satisfying 
\begin{equation}\label{assumption f}
\LpCr{f} \leq A r^{\frac{n}{p} - s }, \qquad\text{for\: any r},\quad 0<r< R/2, 
\end{equation}
where $s > n$ is a non-integral real number and $A$ a positive constant, then there exists $u\in W^{2,p}_{loc}(B_R\setminus\{0\})$ satisfying 
\begin{equation}\label{solution L0}
L_0 u=f\quad\mbox{and}\quad |u(x)|\leq C |x|^{2-s},\qquad\text{for\:any}\quad x\in B_R \setminus\{0\},
\end{equation}
where $C$ depends only on $A$, $s$, $n$, $p$, $R$, $\lambda$, $\Lambda$, and $E$.\\
\end{claim}
\begin{proof}[Proof of Claim \ref{claim 1}] Recall that for $z$ real and $|z| \leq 1$, $|C_j^{\frac{n-2}{2}} (z) | \leq Cj^{n-3}, C = C(n) > 0$ (see \cite{A1}). Therefore, by the Bernstein-Walsh lemma \cite[Section 4.6]{BW}, for any $R_0 > 1$, we have 
\begin{equation}\label{BernsteinWalshResult}
|C_j^{\frac{n-2}{2}} (z) | \leq CR_0^jj^{n-3},\quad\text{for\: any}\quad z \in \mathbb{C},\quad\text{with}\quad |z| \leq (R_0^2 - 1)\big/(2R_0). 
\end{equation}
By \eqref{BoundednessAssumption} -\eqref{EllipticityAssumption2}, there exists $R_0 > 1$, with $R_0$ depending on $\lambda$, $\Lambda$, such that for any $x,y\in\mathbb{R}^n\setminus\{0\}$
\begin{equation}\label{ellipticity R0}
\left|K^{-1}(y) x\cdot y\Big/\left[\left(K^{-1}(y) x\cdot x\right)^{\frac{1}{2}} \left(K^{-1}(y) y\cdot y\right)^{\frac{1}{2}} \right]\right| \leq (R_0^2 - 1)(2R_0),
\end{equation} 
as the term on the left hand side of inequality \eqref{ellipticity R0} is bounded, and 
\begin{equation}\label{ellipticity C}
\left|\left(K^{-1}(y) y\cdot y\right)^{1/2}\right|\Big/\left|\left(K^{-1}(y) x\cdot x\right)^{1/2}\right| \leq \mathcal{C}^{-1} |y|/|x|,
\end{equation}
where $\mathcal{C}>0$ is a constant constant depending on $\lambda, \Lambda$. Hence, for $|y| < \mathcal{C}/R_0 |x|$, the fundamental solution of $L_0$ defined in \eqref{L0} with $y_0=y$, has the expansion
\begin{equation}\label{Gamma expansion}
\Gamma_{K}(x-y)= \left[K^{-1}(y) (x-y)\cdot (x-y)\right]^{(2-n)/2} = \sum_{j=0}^{\infty} P_j(x,y),
\end{equation}
where $P_j(x,y):= \left[K^{-1}(y) y\cdot y\right]^{j/2}\left[K^{-1}(y) x\cdot x\right]^{(2-n-j)/2}\:C_{j}^{(n-2)/2}(\tilde{z}_{x,y})$, for $x\neq 0$. Setting $\nu:=\floor{s}-n>0$, we claim that $\Gamma^{\nu}_{K}(x-y) = \Gamma_{K}(x-y) -  \sum_{j=0}^{\nu} P_j(x,y)$ is a fundamental solution for $L_0$ as well. To show this, it is enough to prove that $L_0 P_j = 0$, for any $j$.  With the linear change of variables $J: \mathbb{C}^n\rightarrow \mathbb{C}^n$, defined by $x\longrightarrow \tilde{x}:=Jx$, where $J$ is the complex-symmetric and invertible matrix such that $JJ= K^{-1}(y)$, we have
\begin{equation}\label{div into laplace}
L_0 P_j = \Delta_{\tx} \left\{\frac{\left(\tilde y\cdot \tilde y\right)^{\frac{j}{2}}}{\left( \tilde x\cdot \tilde x\right)^{\frac{j+n-2}{2}}}\:C_{j}^{\frac{n-2}{2}}\left[\frac{\tilde x\cdot \tilde y}{\left(\tilde x\cdot \tilde x\right)^{\frac{1}{2}} \left(\tilde y\cdot \tilde y\right)^{\frac{1}{2}}}\right] \right\}.  \
\end{equation}
Defining the open and connected set $\mathcal{U} : = \left\{ (\tx, \ty) \in \mathbb{C}^n \times \mathbb{C}^n \; \vert \; \Re ( \tx \cdot \tx) > 0, \quad \Re (\ty \cdot\ty) > 0 \right\}$ and denoting by $\mathcal{U}_{\widetilde{y}}$ its left-hand projection to $\mathbb{C}^n$, for any $\widetilde{y}\in\mathbb{C}^n$, with $\Re (\ty \cdot\ty) > 0$, we have $\mathbb{R}^n\setminus\{0\}\subset \mathcal{U}_{\widetilde{y}}$ and  $J(\mathbb{R}^n\setminus\{0\}) \subset \mathcal{U}_{\widetilde{y}}$, where the latter inclusion is due to the ellipticity assumption \eqref{EllipticityAssumption}. For $x, y \in \mathbb{R}^n \backslash \{0\}$, the right-hand side of \eqref{div into laplace} is zero (see \cite[Lemma~2.3]{A1}). We also have that for any $\widetilde{y}\in\mathbb{C}^n$, with $\Re (\ty \cdot\ty) > 0$, the expression on the right-hand side of \eqref{div into laplace} is analytic in $\widetilde{x}$ on $\mathcal{U}_{\widetilde{y}}$, therefore by analytic continuation, the right hand side of \eqref{div into laplace} is zero also on $\mathcal{U}_{\widetilde{y}}$, hence $L_0 P_j = 0$, for any $j$.\\
We define $C_{R^+} := \{ x \in \mathbb{R}^n \; | \; \frac{\mathcal{C}}{2R_0} |x| < |y| < R \}$, $C_{R^-} := \{ x \in \mathbb{R}^n \; | \; |y| < \frac{\mathcal{C}}{2R_0} |x| \}$ and $C_{\ell} := \{x \in \mathbb{R}^n \; | \; \frac{\mathcal{C}}{R_0}2^{\ell-1} |x| < |y| < \frac{\mathcal{C}}{R_0}2^{\ell} |x| \}$. Assuming without loss of generality that $f\in L^{\infty}(B_R)$ satisfies \eqref{assumption f}, we form
\begin{align}
u(x)&=\int_{C_{R^+}} \Gamma_{K}(x-y) f(y) dy - \sum_{j=0}^{\nu} \int_{C_{R^+}} \!\!\!P_j(x,y) f(y) dy +  \int_{C_{R^-}} \sum_{j=\nu +1}^{\infty} P_j(x,y) f(y) dy \nonumber \\
&:= I_1 + I_2 + I_3. \label{I in 3}
\end{align}
Note that $|I_1| \leq C \int_{C_{R^+}} |x-y|^{2-n} \: |f(y)| dy \leq C|x|^{2-s}$. $I_2$ and $I_3$ can be estimated by extending $f$ outside $B_R$ by setting $f=0$ on $\mathbb{R}^n\setminus B_R$ and by \eqref{BernsteinWalshResult} as
\begin{align}
|I_2|  &\leq  C \sum_{j=0}^{\nu} R_0^j j^{n-3} \mathcal{C}^{-j} \sum_{\ell=0}^{\infty}\int_{C_{\ell}} |f(y)| dy \leq C|x|^{2-s} \sum_{j=0}^{\nu} \frac{ j^{n-3}}{2^j (s-j-n)} \nonumber \\
&\leq C |x|^{2-s}, \label{integral 3}\\ 
|I_3| &\leq  C \hspace{-0.15cm}\sum_{j=\nu + 1}^{\infty} R_0^j j^{n-3} \mathcal{C}^{-j} \sum_{\ell=0}^{\infty}\int_{C_{\ell - 1}}\hspace{-0.15cm} \frac{|y|^j}{|x|^{j+n-2}} |f(y)| dy \leq \hspace{-0.1cm} C |x|^{2-s} \hspace{-0.25cm}\sum_{j=\nu +1}^{\infty}  \frac{j^{n-3}}{2^{j}(j-s+n)} \nonumber \\
&\leq C |x|^{2-s},\label{integral 4}
\end{align}
where $C$ depends only on $n$, $\lambda$, $\Lambda$, and $s$, concluding the proof of the claim.
\end{proof}

By setting $J=\lfloor \frac{m}{\alpha}\rfloor$, with $\alpha$ an irrational number, $0<\alpha<\beta$, by an inductive argument, one can show that for $j=1,\dots , J-1$, by Claim \ref{claim 1}, there are solutions $w_j$ to $L_0 w_j = (L_0 - L)w_{w_{j-1}}$ away from $y_0$, where $||(L_0 - L) W_{J-1}||_{L^p(C_{y_0, r})}\leq Cr^{\frac{n}{p}-s}$, and $s=n+m-(j+2)\alpha>n$, such that $|w_j (x)|\leq C |x-y_0| ^{2-n-m + (j+1)\alpha}$ (see \cite{A1}, \cite{JRCE} for details). Observing that $||(L_0 - L) W_{J-1}||_{L^p(C_{y_0, r})}\leq Cr^{\frac{n}{p}-s}$, with $s=n+m-(J+1)\alpha<n$, we can invoke \cite[Lemma~3.5]{JRCE} to find $W_J$ solution to $L W_J = (L_{0} - L)w_{J-1}$, such that $|W_J(x)|\leq C |x-y_0|^{2-n-m+(J+1)\alpha}$, for any $x\in B_R(y_0)\setminus\{y_0\}$. Setting $w=\sum_{j=0}^J w_j + W_J$ and applying \cite[Lemma~3.4]{JRCE}, we conclude the proof.
\end{proof}


\section{Application to anisotropic time-harmonic Diffuse Optical Tomography}\label{sec:DOTApplication}
We consider the inverse problem in Diffuse Optical Tomography of determining the absorption coefficient in an anisotropic medium $\Omega$, represented by a domain in $\mathbb R^n$, for $n\geq 3$. We assume that the boundary of $\Omega$, $\partial\Omega$, is of Lipschitz class with constants $r_0,L>0$, i.e. that for any $P\in\partial\Omega$ there exists a rigid transformation of coordinates under which we have $P=0$ and $\Omega\cap Q_{r_0}=\{(x',x_n)\in Q_{r_0}\: |\,x_n>\varphi(x')\},$
where $\varphi$ is a Lipschitz function on $B'_{r_0}$ satisfying $\varphi(0)=0 \mbox{  and  }  \|\varphi\|_{C^{0,1}(B'_{r_0})}\leq Lr_0$ (see for example \cite{JRCE}).\\

Under the so-called \textit{diffusion approximation} (\cite{Ar}, \cite{HS}), $\Omega$ is interrogated with an input field that is modulated with a fixed harmonic frequency $\omega=\frac{k}{c}$ , where $c$ is the speed of light and $k$ is the wave number. In this setting, in the anisotropic case, for a fixed $k$, the photon density $u$ in $\Omega$ solves \eqref{Sigma,q equation} with $K$ , $q$ defined by
\begin{equation}\label{KqDef}
K(x) = n^{-1}\left( (\mu_a(x) - ik)I + (I - B(x))\mu_s(x) \right)^{-1}, \quad q(x) = \mu_a(x) - \iu k, \qquad\text{for\:any}\: x\in\Omega,
\end{equation}
respectively. Here $I$ denotes the $n\times n$ identity matrix, $B$ (encompassing the anisotropy of $\Omega$) is known and $I-B$ is positive definite (\cite{Ar}, \cite{HS}).  $\mu_a$ and $\mu_s$ (the optical properties of $\Omega$) are the \textit{absorption} and \textit{scattering coefficients} respectively. Assuming that there are positive constants $\lambda$, $\mathcal{E}$, $E$ and $p > n$ such that 
\begin{eqnarray}
& & \lambda^{-1}\mu_s(x),\; \mu_a(x) \leq \lambda  , \qquad \mbox{ for a.e } x \in\Omega, \label{EllipticityAssumption OT}\\ 
& & \mathcal{E}^{-1} | \xi |^2 \leq (I - B(x)) \xi \cdot \xi \leq \mathcal{E} |\xi |^2, \mbox{ for a.e } x \in \Omega, \mbox{ for any } \xi \in \mathbb{R}^n, \label{AssumptionEllipticIB}  \\
& & ||\mu_a||_{W^{1,\:p}(\Omega)},\quad||\mu_s||_{W^{1,\:p}(\Omega)},\quad ||B||_{W^{1,\:p}(\Omega)} \leq E, \label{AssumptionSobolev}
\end{eqnarray}
$K $, $q$ in \eqref{KqDef} satisfy \eqref{BoundednessAssumption}-\eqref{HolderAssumption} on $\Omega$, hence the resulting system in \eqref{system compact} satisfies the strong ellipticity condition \eqref{strong ellipt} on $\Omega$. Here we assume that $\mu_s$ is known and we address \textit{the inverse problem of determining $\mu_a$ from the Dirichlet-to-Neumann boundary map} defined as follows. We denote $K$ , $q$ with $K_{\mu_a}$, $q_{\mu_q}$ respectively, to emphasise their dependence on the unknown $\mu_a$. 
\begin{defn}\label{DN map}
The Dirichlet-to-Neumann (D-N) map corresponding to $\mu_a$,
\begin{equation*}
\Lambda_{\mu_a}:H^{\frac{1}{2}}(\partial\Omega)\longrightarrow{H}^{-\frac{1}{2}}(\partial\Omega)
\end{equation*}
is defined by 
\begin{equation}
\langle\Lambda_{\mu_a}\:f,\:g\rangle\:=\:\int_{\:\Omega}\Big( K_{\mu_a}(x) \nabla{u}(x)\cdot\nabla\overline{\varphi}(x)+(\mu_a(x)-ik)u(x)\overline{\varphi}(x)\Big)\:dx, 
\end{equation}
for any $f$, $g\in H^{\frac{1}{2}}(\partial\Omega)$, where $u\in{H}^{1}(\Omega)$ is the weak solution to $-\textnormal{div}(K_{\mu_a}(x)\nabla u(x))+ (\mu_a-ik)(x)u(x)=0$ in $\Omega$, 
$u\vert_{\partial\Omega}=f$ in the trace sense and $\varphi\in H^{1}(\Omega)$ is any function such that $\varphi\vert_{\partial\Omega}=g$ in the trace sense. 
\end{defn}
Here $\langle\cdot ,\cdot\rangle$ denotes the dual pairing between $H^{\frac{1}{2}}(\partial\Omega)$ and $H^{-\frac{1}{2}}(\partial\Omega)$ and we will denote $||\Lambda_{\mu_a}||_{\star}= \sup\left\{\left|\langle\Lambda_{\mu_a}\:f,\:g\rangle\right|,\quad f,g\in H^{\frac{1}{2}}(\partial\Omega), ||f||_{H^{\frac{1}{2}}(\partial\Omega)} = ||g||_{H^{\frac{1}{2}}(\partial\Omega)} = 1\right\}$.

\begin{thm}\label{main result}(\textnormal{H\"older stability of boundary derivatives for anisotropic media}).
Let $B$, $\mu_s$ and $\mu_{a_j}$, for $j=1,2$, satisfy \eqref{EllipticityAssumption OT}-\eqref{AssumptionSobolev} and let $K_{\mu_{a_j}}$, $q_{\mu_{a_j}}$ be as in \eqref{KqDef}, for $j=1,2$. If for some integer $h \geq 1$, there is a constant $E_h>0$, such that $||\mu_{a_1} - \mu_{a_2}||_{C^{h,\alpha}(\overline\Omega_r)} \leq E_h$, for all $x \in \overline{\Omega_{r}}$, where $\Omega_r=\left\{x\in\overline\Omega\:|\:\textnormal{dist}(x,\partial\Omega)<r\right\}$, then there is $k_0>0$, such that, for $0 \leq k \leq k_0$, we have
\begin{equation}\label{derivative absorption}
\parallel D^{h}(\mu_{a_{1}}-\mu_{a_{2}})\parallel_{L^{\infty}\:(\partial\Omega)} \leq{C}\parallel\Lambda_{\mu_{a_1}}-\Lambda_{{\mu_{a_2}}}\parallel^{\delta_{h}}_{*}, 
\end{equation}
where $\delta_{h} = \Pi_{i=0}^{h} \frac{\alpha}{\alpha + i}$ and $C$ depends on $n$, $p$, $L$, $r_0$, $diam(\Omega)$, $\lambda$, $E$, $\mathcal{E}$, $h$, $E_h$, $k$ and $k_0$.
\end{thm}

\begin{proof}[Sketch of the proof of Theorem \ref{main result}] The proof follows the same line of that in \cite[Proof of Theorem 2.6]{JRCE}, therefore we only highlight its main steps. The starting point is the well known Alessandrini's Identity \cite{JRCE}, 
\begin{align}
\langle (\Lambda_{\mu_{a_1}} - \Lambda_{\mu_{a_2}} )u_1, u_2\rangle &= \int_{\Omega} \left(K_{\mu_{a_1}} (x) - K_{\mu_{a_2}} (x)\right)\nabla u_1(x)\cdot\nabla \overline{u}_2(x)\:dx \nonumber \\ 
&+\int_{\Omega}\left(\mu_{a_1}(x)-\mu_{a_2}(x)\right)u_1(x)\overline{u}_2(x)\:dx, \label{Alessandrini identity II}
\end{align}
which holds true for any $u_1,u_2 \in H^1 (\Omega)$ solutions to \eqref{Sigma,q equation} in $\Omega$, with $K=K_{\mu_{a_i}}$, $q=q_{\mu_{a_i}}$, for $i=1,2$ respectively. Following the same reasoning of \cite{JRCE}, we feed \eqref{Alessandrini identity II} with the singular solutions $u_i\in W^{2,p}(\Omega)$ in \eqref{sing solution G}, for $i=1,2$ respectively, having a singularity at $z_{\tau}=x^{0}+\tau\widetilde\nu(x^{0})\notin\overline\Omega$, where $x^{0}\in\partial\Omega$ is such that $(-1)^{h}\frac{\partial^h}{\partial\tilde\nu^h}(\mu_{a_1}-\mu_{a_2})(x^{0})\:=\:\left\|\frac{\partial^h}{\partial\tilde\nu^h}\left(\mu_{a_1}-\mu_{a_2}\right)\right\|_{L^{\infty}(\partial\Omega)}$ and $\widetilde\nu$ is a $C^{\infty}$, non-tangential vector field on $\partial\Omega$ pointing outwards of $\Omega$ (see \cite{JRCE}) so that $C\:\tau\leq{d}(z_{\tau},\:\partial\Omega)\leq\tau$, for any $\tau$, $0\leq\tau\leq\tau_{0}$, with $\tau_{0}, C>0$ depending on $L$, $r_0$ only. Observing  also that for $k=0$ (see \cite[Lemma 3.1]{A1}, \cite[Lemma 3.7]{JRCE})
\begin{equation}\label{estimate gradient}
|Du_i(x)|> |x-z_{\tau}|^{1-(n+m)},\qquad\text{for\:every}\: x,\quad 0 < |x-z_{\tau}| < r_0,
\end{equation}
where $r_0$ depends only on $\lambda$, $E$, $p$, $m$ and $R$, we argue that \eqref{estimate gradient} still holds true if $0\leq k\leq k_0$, for some $k_0>0$. We conclude the proof, as in \cite{JRCE}, by allowing $z_{\tau}$ to approach $x^0\in\partial\Omega$, by letting $\tau\rightarrow 0$.
 \end{proof}

\section*{Acknowledgements}
All authors were partly supported by Science Foundation Ireland under Grant number 16/RC/3918. RG, CN also acknowledge the support of the Isaac Newton Institute for Mathematical Sciences, Cambridge, where part of the research of this article was carried out during a semester on inverse problems in the spring 2023.\\

\end{document}